\documentclass[12pt]{amsart}
\usepackage{amsmath,amscd,amssymb,amsfonts}
\usepackage[all]{xy}

\newcommand{\ver}{September 6, 2008}
\newcommand{\bC}{{\mathbb C}}

\newcommand{\bN}{{\mathbb N}}
\newcommand{\bP}{{\mathbb P}}

\newcommand{\bQ}{{\mathbb Q}}
\newcommand{\bZ}{{\mathbb Z}}
\newcommand{\cA}{{\mathcal A}}

\newcommand{\cE}{{\mathcal E}}

\newcommand{\cG}{{\mathcal G}}
\newcommand{\cH}{{\mathcal H}}

\newcommand{\cJ}{{\mathcal J}}
\newcommand{\cO}{{\mathcal O}}

\newcommand{\cS}{{\mathcal S}}

\newcommand{\wti}{\widetilde}

\newcommand{\codim}{\hbox{\rm codim}\,}

\newcommand{\ra}{\rightarrow}
\newcommand{\rndown}[1] {\llcorner {#1} \lrcorner}
\newcommand{\rndup}[1] {\ulcorner {#1} \urcorner}

\theoremstyle{plain}
\newtheorem{thm}{Theorem}[section]
\newtheorem{cor}[thm]{Corollary}

\newtheorem{lem}[thm]{Lemma}
\newtheorem{prop}[thm]{Proposition}

\theoremstyle{definition}
\newtheorem{df}[thm]{Definition}
\newtheorem{rem}[thm]{Remark}

\title[Jumping numbers of hyperplane arrangements]{Jumping numbers of hyperplane arrangements}
\author{Nero Budur}
\address{Department of Mathematics, The
University of Notre Dame, IN 46556, USA} \email{nbudur@nd.edu}

\date{\ver}

\thanks {The author was supported by the NSF grant DMS-0700360.}

\begin{document}

\maketitle

\begin{abstract}
M. Saito \cite{MS} proved that the jumping numbers of a
hyperplane arrangement depend only on the combinatorics of the
arrangement. However, a formula in terms of the combinatorial data was still
missing. In this note, we give a formula and a different proof of the fact
that the jumping numbers of a hyperplane arrangement depend only
on the combinatorics. We
also give a combinatorial formula for part of the Hodge spectrum
and for the inner jumping multiplicities.

\end{abstract}

\section{Introduction}

Jumping numbers are numerical measures of the complexity of the singularities of a variety (see section \ref{sec. review}). M. Saito \cite{MS} proved that the jumping numbers of a reduced
hyperplane arrangement depend only on the combinatorics of the
arrangement, answering a question of M. Musta\c{t}\u{a} \cite{Mu}. The method of his proof was by reduction to the
corresponding statement about the Hodge spectrum. His proof
extends to non-reduced arrangements as well by taking into account
the multiplicities along the hyperplanes in the arrangement.
However, a formula in terms of the combinatorial data was still
missing.

In this note, we give a formula and a different proof of the fact
that the jumping numbers of a hyperplane arrangement depend only
on the combinatorics and the multiplicities along hyperplanes. We
also give a combinatorial formula for part of the Hodge spectrum
and for the inner jumping multiplicities. Combinatorial formulas for those jumping numbers which change the support of the multiplier ideals have been obtained in \cite{Mu}- Example 2.3, for reduced arrangements, and refined by \cite{Te}- Remark 3.2.

 Let $\cA$ be a central hyperplane
arrangement in $\bC^n$. Denote the intersection lattice of $\cA$
by $L(\cA)$, that is the set of subspaces of $\bC^n$ which are
intersections of subspaces $V\in\cA$. We consider the
corresponding arrangement of projective hyperplanes in
$Y=\bP^{n-1}$ given by $\bP(V)$ for $V\in \cA$. Let $D$ be an
effective divisor on $Y$ supported on $Supp \,(D)=\cup_{V\in
\cA}\bP(V)$. We assume that $Supp\,(D)$ is the compactification of
a central hyperplane arrangement in some $\bC^{n-1}\subset Y$. For
our purposes, the general case can be reduced to this particular
case.

We will give a combinatorial criterion, in terms of $L(\cA)$ and
the multiplicities of $D$, for a positive rational number to be a
jumping number of $D$ in $Y$. It is known that $1$ is trivially a
jumping number of $D$ and that $c>1$ is a jumping number if and
only if $c-1$ is. Thus it is
enough to determine which $c\in (0,1)$ are jumping numbers of $D$.

Let $\cG'\subset L(\cA)-\{\bC^n\}$ be a building set (see
\cite{DP}-2.4 or \cite{Te}-Definition 1.2). Let
$\cG=\cG'\cup\{0\}$. For simplicity, one can stick with the
following example for the rest of the article:
$\cG=L(\cA)\cup\{0\}-\{\bC^n\}$, when $\cG'$ is chosen to be
$L(\cA)-\{\bC^n\}$. The advantage of considering smaller building
sets is that computations might be faster (see \cite{Te}-Example
1.3-(c)).

 For $V\in \cG$, define
$r(V)=\codim (V)$, $\delta(V)=\dim V$, and $$s(V)=\sum_ {V\subset
W\in \cA }\text{mult} _{\bP(W)}(D).$$ Set
$d=\sum_{V\in\cA}\text{mult}_{\bP(V)}(D)$ and
$$a_0=\max\left\{d-n+1,\sum_ {W\in \cG-\{0\}}\max\{0, s(W)-r(W)\}\right\}.$$
For any finite set $\cS$, set $|\cS |$ to be the number of elements of $\cS$. For a rational number $c$ let
$$\cS_c=\{\ V\in \cG-\{0\}\ |\ cs(V)\in\bZ\ \}.$$
For $V\in\cG$ let
$$
a_V(c)=\left\{
\begin{array}{ll}
r(V)-1-\rndown{cs(V)} &  \text{ if }V\in \cG-\{0\}, \notin \cS_c,\\
r(V)-cs(V) & \text{ if }V\in \cG-\{0\}, \in\cS_c,\\
-a_0 & \text{ if }V=0,
\end{array}\right.
$$

For a nonempty nested subset $\cS$ of $\cG-\{0\}$ and for
$V\in\cS\cup\{\bC^n\}$, denote by $V_\cS$ the subspace $\sum W$
where the sum is over $W\varsubsetneq V$ such that $W\in\cS$. In
other words, $V_\cS$ is the maximal element of $\cS$ which is
$\varsubsetneq V$. Set $V_\cS=0$ if there is no such maximal
element. Let $Q(x)=x/(1-\exp(-x))$ considered as an element of the
formal power series ring $\bQ[[x]]$.

\begin{df}\label{def. poly P} Let $\cS$ be a nonempty nested subset of
$\cG-\{0\}$ and let $V\in\cS\cup\{\bC^n\}$. For $W\in\cG$ with
$V_\cS\subset W\varsubsetneq V$ define a formal power series
$P_W^{\cS , V}\in\bQ [[c_{W'}]]_{W'\in \cG}$ as follows. If
$W=V_\cS$ set
$$P_W^{\cS , V}= Q(-\mathop{\sum_{W'\subset
V_\cS}}_{ \{W'\}\cup\cS\subset\cG\text{ nested} }
c_{W'})^{\delta(V)-\delta(V_\cS)}.
$$
If $W\ne V_\cS$ define
\begin{align*}
P_W^{\cS , V} & = Q(-\mathop{\sum_{W'\varsubsetneq W}}_{
\{W'\}\cup\cS\subset\cG\text{ nested}}
 c_{W'})^{-(\delta(V)-\delta(W))} \cdot Q(c_W)\cdot \\
 & \cdot Q(-\mathop{\sum_{ W'\subset
W}}_{ \{W'\}\cup\cS\subset\cG\text{ nested}
}c_{W'})^{\delta(V)-\delta(W)}.
\end{align*}
\end{df}

\begin{df}\label{def. T_j^S} Let $\cS$ be a nonempty nested subset of
$\cG-\{0\}$. Let $0\le j\le n-1-|\cS|$. Define the polynomial
$T_j^\cS\in \bQ[c_V]_{V\in\cG}$ to be the homogeneous part of
degree $j$ of the formal power series
$$T^\cS:=\prod_{V\in\cS\cup\{\bC^n\}} \ \ \mathop{\prod_{V_\cS\subset W\varsubsetneq
V}}_{W\in\cG} P^{\cS,V}_W.$$
\end{df}

Let $I\subset\bZ[c_V]_{V\in\cG}$ be the ideal of \cite{DP}-5.2
(for the projective case, see Remark \ref{rem. isom}). Recall that $I$ depends only on $\cG$
and that $\bZ[c_V]_{V\in\cG}/ I$ is isomorphic to the cohomology
ring of the canonical log resolution in terms of $\cG$ of $(Y,D)$,
i.e. the wonderful model of \cite{DP}. More precisely, $I$ is
generated by two types of polynomials:
\begin{equation}\label{eq, type 1}
\prod_{V\in \cH}c_V
\end{equation}
if $\cH\subset\cG$ is not a nested subset, and by
\begin{equation}\label{eq, type 2}
\prod_{V\in \cH}c_V\left (  \sum_{W'\subset W} c_{W'} \right
)^{d_{\cH,W}},
\end{equation}
where $\cH\subset \cG$ is a nested subset, $W\in\cG$ is such that
$W\varsubsetneq V$ for all $V\in\cH$, and
$d_{\cH,W}=\delta(\cap_{V\in\cH} V) -\delta (W)$. In (\ref{eq,
type 2}), one considers $\cH=\emptyset$ to be nested, in which
case (\ref{eq, type 2}) is defined for every $W\in\cG$ by setting
$\delta (\emptyset)=n$.

\begin{thm}\label{thm. jumping numbers} With the notation as
above, a rational number $c\in (0,1)$ is a jumping number of
$D\subset Y$ if and only if
$$
\mathop{\sum_{\text{nested}}}_{{\emptyset}\ne\cS\subset \cS_c}
\sum_{0\le j}^{n-1-|\cS|} \frac{(-1)^{|\cS|+1}}{j!} \left (
\sum_{V\in\cG} a_V(c)c_V \right )^{j} T_{n-1-|\cS|-j}^\cS
\prod_{V\in\cS} c_V
$$
does not belong to the ideal $I\subset \bQ[c_V]_{V\in\cG}$.
\end{thm}

Since we are assuming that $D$ is the compactification of a
central hyperplane arrangement in $\bC^{n-1}$, let $x\in Y$ be the
point corresponding to the origin of $\bC^{n-1}$. As for jumping
numbers, the method of the proof of Theorem \ref{thm. jumping
numbers} gives a formula in terms of combinatorics for the inner
jumping multiplicities $n_{c,x}(D)$ of a positive rational number
$c$ along $D$ at the point $x$ (see section \ref{sec. review}).

\begin{thm}\label{thm. inner jumping numbers}  With the notation as above, let $c$ be a positive rational number.
Then the inner jumping multiplicity of $c$ along $D$ at $x$ is $0$
if there are no subspaces $V\in\cG$ with $\delta (V)=1$ or if
$cd\not\in \bZ$. Otherwise, let  $V_x\in\cG$ be the only subspace
with $\delta=1$, that is $\bP(V_x)=\{x\}$. Then
$$n_{c,x}(D)= \sum_{0\le j\le n-2}\frac{1}{j!}\left (
\sum_{V\in\cG-\{0\}} a_V(c)c_V \right
)^{j}T_{n-2-j}^{\{V_x\}}c_{V_x},
 $$
where the right-hand side is viewed as a number via identification
of the degree $n-1$ homogeneous part of $\bQ[c_V]_{V\in\cG}/I$
with $\bQ\cdot (-c_0)^{n-1}$.
\end{thm}

By a result of \cite{Bu} (see also \cite{BS}), for $c\in (0,1]$
the inner jumping multiplicities $n_{c,x}(D)$ are the
multiplicities of $c$ in the Hodge spectrum of $D$ at $x$
(\cite{St}). Thus we have a combinatorial formula for the
beginning part of the Hodge spectrum of a central hyperplane
arrangement.

In section 2 we review multiplier ideals and intersection theory.
In section 3 we set the problem into global setting, in
preparation for using the Hirzebruch-Riemann-Roch theorem. In
section 4, we prove Theorems \ref{thm. jumping numbers} and
\ref{thm. inner jumping numbers} via Hirzebruch-Riemann-Roch on
wonderful models. In the last section we give examples
illustrating how Theorems \ref{thm. jumping numbers} and \ref{thm.
inner jumping numbers} work.

In this article, inclusion of sets is denoted by $\subset$ and
strict inclusion of sets is denoted by $\subsetneq$.

We would like to thank M. Saito, to whom we are indebted for the
proof of Lemma \ref{lema linear forms}, for sharing with us the
results of his preprint \cite{MS} which was the inspiration behind
this article, and for many comments. We also thank M.
Musta\c{t}\u{a}, M. Schulze, and Z. Teitler for useful
discussions. The author was supported by the NSF grant
DMS-0700360.

\medskip

\begin{section}{Review of multiplier ideals, intersection
theory}\label{sec. review}

The notation of the current section is independent of the rest of
the article.

\medskip
\noindent {\bf Multiplier ideals.} We review some basic facts from
the theory of multiplier ideals (see \cite{La}- Chapter 9). Let
$Y$ be a smooth complex variety. Let $D$ be an effective
$\bQ$-divisor on $Y$. Let $\rho: Y' \ra Y$ be a log resolution of
$(Y,D)$ and let $K_{Y'/Y}$ be the relative canonical divisor. The
{\it multiplier ideal} of $D$ is the ideal sheaf
$$\cJ(D) = \rho_* \cO_{Y'}(K_{Y'/Y}-\rndown{\rho^* D})\ \ \ \subset
\cO_Y.$$ The choice of log resolution does not matter in the
definition of the $\cJ(D)$ and one can extend the definition by
allowing, instead of $D$, any finite formal linear combination of
subschemes of $Y$ with positive coefficients. A positive rational
number $c$ is called a {\it jumping number} of $D$ if $\cJ(c\cdot
D)\ne \cJ((c-\epsilon)\cdot D)$ for all $0< \epsilon \ll 1$. It is
known that a positive rational number $c$ is a jumping number if
and only if $c+1$ is a jumping number (\cite{La}- Example 9.3.24).
Let $x$ be a point in the support of $D$ and let $c>0$. The {\it
inner jumping multiplicity} of $c$ along $D$ at $x$ (\cite{Bu}-
Definition 2.4) is defined as
$$n_{c,x}(D)=\dim_\bC \frac{\cJ((c-\epsilon)D)}{\cJ((c-\epsilon)D+\delta
\{x\})}\ ,$$ where $0<\epsilon\ll\delta\ll 1$. By
\cite{Bu}-Proposition 2.8, if the inner jumping multiplicity of
$c$ is nonzero then $c$ is a jumping number.

\begin{thm}\label{thm local vanishing} (Local vanishing, \cite{La}- Theorem
9.4.1). With the notation as above,
$$R^j\rho_*\cO_{Y'}(K_{Y'/Y}-\rndown{\rho^* D}) =0\ \ \text{ for
}j>0.$$
\end{thm}

\begin{thm}\label{thm Nadel vanishing} (Nadel vanishing, \cite{La}- Theorem
9.4.9). With the notation as above, assume in addition that $Y$ is
projective. Let $L$ be any integral divisor such that $L-D$ is nef
and big. Then
$$H^i(Y,\cO_Y(K_Y+L)\otimes\cJ(D))=0\ \ \ \text{ for }i>0.$$
\end{thm}

\medskip

\noindent {\bf Intersection theory.} We recall some facts about
intersection theory (see \cite{Fu}). Let $Y$ be a smooth
projective complex variety. For a vector bundle, or locally free
$\cO_Y$-module of finite rank, $\cE$ on $Y$, we denote by
$c_j(\cE)$ the image of the $j$-th Chern class of $\cE$ in
$H^{2j}(Y,\bZ)$. The total Chern class is defined to be
$c(\cE)=\sum_j c_j(\cE)$ in the cohomology ring $H^*(Y,\bZ)$. The
roots $x_i$ of $\cE$ are formal symbols satisfying the formal
decomposition $\sum_j c_j(\cE)t^j=\prod_i (1+x_it)$. Then one
defines $ch(\cE)=\sum _i\exp (x_i)$, and writes $ch (\cE)=\sum_j
ch_j(\cE)$ with $ch_j(\cE)\in H^{2j}(Y,\bQ)$. The Todd class of
$\cE$ is defined as $td (\cE)=\prod Q(x_i)$, where
$Q(x)=x/(1-\exp(-x))$. The Todd class of $Y$ is denoted by $Td(Y)$
and is defined as the Todd class of the tangent bundle of $Y$. One
writes $Td(Y)=\sum_j Td_j(Y)$ where $Td_j(Y)\in H^{2j}(Y,\bQ)$.

\begin{thm}\label{thm HRR} (Hirzebruch-Riemann-Roch, \cite{Fu}- Corollary
15.2.1) Let $\cE$ be a vector bundle on a smooth projective
complex variety $Y$. Then $\chi (Y,\cE)$ is the intersection
number $\sum_{i+j=\dim Y} ch_i(\cE)\cdot Td_j(Y)$.
\end{thm}

Let $X_1,\ldots X_t$ be disjoint smooth subvarieties of $Y$ of
codimension $d$. Let $\rho:\wti{Y}\ra Y$ be the blow up of
$\coprod X_i$. Let $E_i$ be the exceptional divisor on $\wti{Y}$
corresponding to $X_i$. Let $[E_i]\in H^2(\wti{Y},\bZ)$ be the
cohomology class of $E_i$. Let $N_i$ be the
normal bundle of $X_i$ in $Y$. Suppose there exist $c_{k,i}\in
H^{2k}(Y,\bZ)$ such that the Chern classes $c_k(N_i)$ are the
restriction of $c_{k,i}$ to $X_i$. The following computes the
total Chern class of $\wti{Y}$ and follows from \cite{Fu}-Example
15.4.2.

\begin{prop}\label{prop. chern classes blow up} With the notation
as above,
$$c(\wti{Y})=\rho^* c(Y)\prod_{1\le j\le t}\left \{ \left (  \sum_{0\le k\le d} \rho^*c_{k,j}  \right )^{-1}(1+[E_j])
\left (  \sum_{0\le i\le d} (1-[E_j])^i\rho^* c_{d-i,j}  \right )
\right \}.$$
\end{prop}

\end{section}

\medskip

\begin{section}{Uniform bound for jumps in multiplier ideals}

\noindent {\bf  Affine case.} Let $\cA'$ be a central hyperplane arrangement in $\bC
^{n-1}$. Let $D'$ be an effective divisor on $\bC ^{n-1}$ with support
$\cA'$. Let $L(\cA')$ be the intersection lattice of $\cA'$. For
$V\in L(\cA')$, define $r'(V)=\codim (V)$ and $s'(V)=\sum_ {V\subset
W\in \cA' }\text{mult} _W(D')$. Let $\cG'\subset L(\cA')-\{\bC ^{n-1}\}$
be a building set.
Recall the following result of M. Musta\c{t}\u{a}
\cite{Mu}-Corollary 2.1 for the case of reduced arrangements, and
refined by Teitler \cite{Te}-Theorem 1.4.

\begin{prop}\label{prop characterization of jump. nos.} If $D'$ is an effective divisor supported on a central hyperplane arrangement in $\bC^{n-1}$, then
$$\cJ(cD')=\bigcap_{W\in\cG'}I_W^{\ \rndup{cs'(W)}-r'(W)}.$$
Moreover, $c$ is a jumping number of $D'$ if and only if there are $V\in\cG'$
and $m\in\bN$ such that $c=\frac{r'(V)+m}{s'(V)}$ and such that
$$\bigcap_{V\subset W\in\cG'} I_W^{\ \rndup{cs'(W)}-r'(W)}\not\subset I_V^{m+1}.$$
\end{prop}

The following lemma will allow us to bound the degrees of the polynomials at which we need to look to detect a jump of multiplier ideals. We have conjectured the statement, proved some cases, and M. Saito proved it in general.

\begin{lem}\label{lema linear forms} For $1\le
i\le s$, let $I_i\subset {\bf{C}}[x_1,\ldots,x_n]$ be ideals
generated by linear forms. Suppose $I_1^{a_1}\cap \ldots\cap
I_s^{a_s}\not\subset I_1^{a_1+1}$ for some positive integers $a_i$.
Then there exists $f$ in $I_1^{a_1}\cap \ldots\cap I_s^{a_s}$ but
not in $I_1^{a_1+1}$ of degree at most $a_1+\ldots +a_s$.
\end{lem}
\begin{proof} The following short and elementary proof of this lemma is due M. Saito who kindly allowed us to reproduce it here. After a change of coordinates, we can assume that $I_1=(x_1,\ldots ,x_m)$ for some $m\le n$. After reordering of indices, we can assume that there is $r\in\{1,\ldots ,s\}$ such that $I_i\subset I_1$ for $1\le i\le r$ and $I_i\not\subset I_1$ for $r<i\le s$. Let $J_i=I_i\cap\bC[x_1,\ldots ,x_m]$. Then
$$\bigcap_{1\le i\le r}I_i^{a_i}=\bigcap_{1\le i\le r}J_i^{a_i}\;\cdot\;\bC[x_1,\ldots ,x_n].$$
Since $\cap_{1\le i\le r}I_i^{a_i}\not\subset I_1^{a_1+1}$, we have that $\cap_{1\le i\le r}J_i^{a_i}\not\subset J_1^{a_1+1}$. The ideals $J_i$ are homogeneous. Hence we can find a homogeneous polynomial $u$ in $\cap_{1\le i\le r}J_i^{a_i}$ which does not belong to $J_1^{a_1+1}=(x_1,\ldots ,x_m)^{a_1+1}$. Then the degree of $u$ must be $a_1$. For $r<i\le s$, take $v_i\in I_i$ but $\not\in I_1$ to be a linear form. Let $f=u\prod_{r<i\le s}v_i^{a_i}$. Then $f\in\cap_{1\le i\le s}I_i^{a_i}$, but $f\not\in I_1^{a_1+1}$, and the degree of $f$ is $a_1+a_{r+1}+\ldots +a_s$.
\end{proof}

Let $a_0'=\sum_{W\subset\cG'}\max \{0,s'(W)-r'(W)\}$. By Proposition \ref{prop characterization of jump. nos.} and Lemma \ref{lema linear forms}, we have:

\begin{cor}\label{cor. affine case} If $D'$ is an effective divisor supported on a central hyperplane arrangement in $\bC^{n-1}$, then $c\in (0,1)$ is a jumping number of $D'$ if and only if there exists $f\in\bC[x_1,\ldots ,x_{n-1}]$ of degree at most $a_0'$ with $f\in\cJ((c-\epsilon)D')$ for $0<\epsilon\ll 1$, but $f\not\in \cJ(cD')$.
\end{cor}

\medskip
\noindent {\bf Projective case.} Let $\cA$ be a central hyperplane arrangement in $\bC^n$. Denote the
intersection lattice of $\cA$ by $L(\cA)$. We consider the corresponding
arrangement of projective hyperplanes in $Y=\bP^{n-1}$ given by
$\bP(V)$ for $V\in \cA$. Let $D$ be an effective divisor on $Y$
supported on $\cup_{V\in \cA}\bP(V)$. Assume that the support of $D$ is the compactification of a central hyperplane arrangement in some $\bC^{n-1}\subset Y$. Let $\cG'\subset L(\cA)-\{\bC^n\}$ be a building set and let $\cG=\cG'\cup\{0\}$.  For $c$ a positive real number, let $\cJ
(c D)$ be the multiplier ideal of $cD$ in $Y$. Let $\cG (c D)=\cJ
((c-\epsilon) D)/\cJ (c D)$ for $0<\epsilon\ll 1$. Thus $c$ is a jumping number of $D$ if and only if $\cG(cD)\ne 0$. Recall that we defined in the introduction, for $V\in \cG-\{0\}$, the numbers $r(V)$ and $s(V)$. Let $a_0$ be defined as in the introduction. By Corollary \ref{cor. affine case}, we have:

\begin{cor}\label{cor. reduction to global invariants}  For all $c\in (0,1)$,
$$\cG(c D)\ne 0 \Leftrightarrow H^0(Y,\cO_Y(a_0)\otimes\cG (c D))\ne 0.$$
\end{cor}

\end{section}

\medskip
\begin{section}{Intersection theory on canonical log
resolutions.}

\medskip

\noindent {\bf The canonical log resolution.} Let $\cA$ be a
central hyperplane arrangement in $\bC^n$. We consider the
corresponding arrangement of projective hyperplanes in
$Y=\bP^{n-1}$ given by $\bP(V)$ for $V\in \cA$. Let $D$ be an
effective divisor on $Y$ supported on $\cup_{V\in \cA}\bP(V)$. We
assume also that the support of $D$ is the compactification of a
central hyperplane arrangement in some $\bC^{n-1}\subset Y$. Let
$\cG'\subset L(\cA)-\{\bC^n\}$ be a building set. Let
$\cG=\cG'\cup\{0\}$. For example, $\cG=L(\cA)\cup\{0\}-\{\bC^n\}$.

We consider the canonical log resolution $\rho:\wti{Y}\ra Y$ of
$D$ obtained from succesive blowing ups of the (disjoint) unions
of (the proper transforms) of $\bP(V)$ for $V\in \cG-\{0\}$ of
same dimension. This is the so-called wonderful model of
\cite{DP}- section 4. More precisely, $\rho$ and $\wti{Y}$ are
constructed as follows.

The following notation is taken from \cite{MS}- section 2. Let
$Y_0=Y$. Let $C_0$ be $\bP(V)$ for
$V\in\cG-\{0\}$ with $\delta (V)=1$ (there is at most one such $V$, by assumption). Let $\rho_0:Y_1\ra Y_0$ be
the blow up of $C_0$. Then $\rho_i$ and $Y_{i+1}$ are constructed
inductively as follows. Let $C_i\subset Y_{i}$ be the disjoint
union of the proper transforms, under the map $\rho_{i-1}$, of
$\bP(V)$ for $V\in \cG-\{0\}$ with $\delta(V)=i+1$.  Let
$\rho_i:Y_{i+1}\ra Y_i$ for $0\le i <n-2$ be the blow up of $C_i$.
Define $\wti{Y}=Y_{n-2}$ and $\rho$ as the composition of the
$\rho_i$.

We need some more notation, also from \cite{MS}- section 2. Let $C_{V,0}=\bP(V)\subset Y_0$.
For $V\in \cG-\{0\}$ with $\delta(V)=i+1$, $C_{V,j}$ denotes the
proper transform of $C_{V,0}$ in $Y_j$ for $1\le j\le i$. Let
$E_{V,i+1}$ be the exceptional divisor in $Y_{i+1}$ corresponding
to $C_{V,i}$. Let $E_{V,j}$ be the proper transform of $E_{V,i+1}$
in $Y_j$ for $i+1<j\le n-2$. On $\wti{Y}$, let $E_V=E_{V,n-2}$ if
$\delta (V)<n-1$, and $E_V=C_{V,n-2}$ if $\delta (V)=n-1$. Also
let $E_{0,i}$ ($0\le i\le n-2)$, and $E_0$, denote the proper
transform in $Y_i$, respectively in $\wti{Y}$, of a general
hyperplane of $Y=\bP^{n-1}$. Denote by $[E_V]$ the cohomology
class of $E_V$ on $\wti{Y}$, where it will be clear from context
what coefficients (integral, rational) we are considering.

For any subset $\cS$ of $\cG-\{0\}$, set $E^{\cS}=\cup_{V\in
\cS}E_V$ and $E_{\cS}=\cap_{V\in\cS}E_V$. For a rational number
$c$, recall the definitions of $\cS_c$, $a_0$, and $a_V(c)$ from
the introduction. Also define $a_V'(c)$ to equal $a_V(c)$ for
$c\ne 0$ and, otherwise, $a_V'(0)=a_0$.

\begin{lem}\label{lemma reduction to euler char} With the notation as above,
$$H^0(Y,\cO_Y(a_0)\otimes\cG (c D))=\chi\left (\cO_{E^{\cS_c}}\left (\sum_{V\in\cG}a_V'(c)E_V\right )\right ).$$
\end{lem}
\begin{proof} We have that $K_{\wti{Y}/Y}=\sum_{V\in
\cG-\{0\}}(r(V)-1)E_V$ and $\rho^*(D)=\sum_{V\in
\cG-\{0\}}s(V)E_V$ (\cite{Te}-Lemma 2.1). Then, from the
definition of multiplier ideals and Theorem \ref{thm local
vanishing}, we have
\begin{align*}
\cG(cD) &=\rho_*(\cO_{E^{\cS_c}}(\sum_{V\in\cG-\{0\}}a_V(c)E_V)),\text{       and}\\
0 &=R^i\rho_*(\cO_{E^{\cS_c}}(\sum_{V\in\cG-\{0\}}a_V(c)E_V)) \text{       for }i>0.
\end{align*}
We can rewrite $\cO_Y(a_0)$ as $\omega_Y\otimes\cO_Y(a_0+n)$. By
definition, $a_0+n>d$. Hence Theorem \ref{thm Nadel vanishing}
applies and we have
\begin{align*}
H^0(Y,\cO_Y(a_0)\otimes\cG (c D)) &=\chi\left (\cO_{\wti{Y}}(a_0E_0)\otimes\cO_{E^{\cS_c}}(\sum_{V\in\cG-\{0\}}a_V(c)E_V)\right )\\
&= \chi\left (\cO_{E^{\cS_c}}\left (\sum_{V\in\cG}a_V'(c)E_V\right
)\right ).
\end{align*}

\end{proof}

\begin{lem}\label{lema after mayer-vietoris}
With the notation as in Lemma \ref{lemma reduction to euler char},
a rational number $c\in (0,1)$ is a jumping number of $D$ if and
only if
$$\mathop{\sum_{\emptyset \ne\cS\subset\cS_c}}_{\text{nested}}(-1)^{|\cS|+1}\chi\left (\cO_{E_{\cS}}\left (\sum_{V\in\cG}a_V'(c)E_V\right )  \right
)\ne 0.$$
\end{lem}
\begin{proof} Follows from Lemma \ref{lemma reduction to euler
char} and Corollary \ref{cor. reduction to global invariants} via the Mayer-Vietoris exact sequence
$$0\ra\cO_{E^{\cS_c}}\ra\mathop{\bigoplus_{\cS\subset\cS_c}}_{
|\cS|=1}\cO_{E_\cS}\ra \mathop{\bigoplus_{\cS\subset\cS_c}}_{
|\cS|=2}\cO_{E_\cS}\ra\ldots\ra\cO_{E_{\cS_c}}\ra 0.$$ The
intersection $E_\cS$ is nonempty if and only if $\cS$ is nested
(\cite{MS}-2.7, \cite{DP}-4.2).
\end{proof}

Next goal is to compute $\chi\left (\cO_{E_{\cS}}\left
(\sum_{V\in\cG}a_V'(c)E_V\right )\right )$ via
Hirzebruch-Riemann-Roch.

\begin{rem}\label{rem. isom} Let $I\subset\bZ[c_V]_{V\in\cG}$ be the ideal of
\cite{DP}-5.2 described in the introduction. By loc. cit. there is
an isomorphism
\begin{align}\label{eq. isom cohomology}
\bZ[c_V]_{V\in\cG} / I\   & \mathop{\longrightarrow}^{\sim}\
H^*(\wti{Y},\bZ)\  \mathop{\longleftarrow}^{\sim}\
\bZ[[c_V]]_{V\in\cG} / I \\
\notag & 1\mapsto [\wti{Y}],\\
\notag & c_V  \mapsto [E_V]\ \ \ \ \text{ if }V\ne 0,\\
\notag & c_0 \mapsto -[E_0].
\end{align}
Indeed, this follows from \cite{DP}-5.2 Theorem, \cite{DP}-4.1
Theorem, part (2), and \cite{DP}-4.2 Theorem, part (4). The only
case left out by \cite{DP}-4.2 Theorem, part (4) is the one
corresponding with $E_0$ in our notation. But this follows from
the fact that, in their notation, the linear equivalence class of
$D_{V^*}$ restricted to $D_{V^*}$  is the negative of the class of
the proper transform in $D_{V^*}$ of a general hyperplane in the
exceptional divisor of the blowup of the origin of $V$. The
objects $V$ and $D_{V^*}$ of \cite{DP} correspond to $\bC^n$ and,
respectively, $\wti{Y}$, in our notation. The exceptional divisor
of the blowup of the origin of $V$ is, in our notation,
$\bP^{n-1}$, the ambient space of our projective arrangement of
hyperplanes.
\end{rem}

\begin{lem}\label{lem Computation of chi E_S} With the notation as in Theorem \ref{thm. jumping numbers},
let $\emptyset\ne\cS$ be a nested subset of $\cG-\{0\}$, and $c$ a
rational number. Then

\begin{align*}
\chi\left (\cO_{E_{\cS}} \left
(\sum_{V\in\cG}a_V'(c)E_V\right )\right ) &=\\
=  \sum_{0\le j}^{n-1-|\cS|} \frac{1}{(n-1-|\cS|-j)!} & \left (
\sum_{V\in\cG} a_V(c)c_V \right )^{n-1-|\cS|-j} T_j^\cS
\prod_{V\in\cS} c_V,
\end{align*}
where the right-hand side is viewed as an intersection number via
the isomorphism (\ref{eq. isom cohomology}).

\end{lem}

\medskip
\noindent {\bf Proof of Theorem \ref{thm. jumping numbers}.} It
follows from Lemma \ref{lema after mayer-vietoris} and
Lemma \ref{lem Computation of chi E_S}.\ \ $\Box$

\medskip

Before we prove Lemma \ref{lem Computation of chi E_S}, we need some preliminary results.

Write $Y=\bP(\bC^n)$ and $\wti{Y}=\bP(\bC^n)^\cG$. This notation
makes sense if one replaces $\bC^n$ and $\cG$ by any vector space
with a finite set of proper vector subspaces which is closed under
intersections and contains $\{0\}$. For a nested subset $\cS\subset
\cG-\{0\}$ and $V\in\cS\cup \{\bC ^n\}$, let $V_\cS$ be as in introduction. Define
$\bC^\cS _V=V/V_\cS$
and set
$$\cG^\cS _V=\{\ W'\subset\bC^\cS _V\ |\ W' \text{ is the image of }W\text{ in }\bC^\cS
_V\text{ for some }W\in\cG, W\subsetneq V \}.$$ We have the following
description of $E_\cS$ (\cite{MS}-2.7, \cite{DP}-4.3):

\begin{prop}\label{prop. decomposition of canonical resolution}
With the notation as above, let $\cS\subset \cG-\{0\}$ be a nested
subset. Then
$$E_\cS=\prod_{V\in\cS\cup\{\bC^n\}}\bP(\bC^\cS _V)^{\cG^\cS
_V}.$$
\end{prop}

By \cite{Fu}- Example
15.2.12, the Todd class of $E_\cS$ is also a product:

\begin{lem}\label{lem. todd class of E_S} With the notation as in Proposition \ref{prop. decomposition of canonical resolution},
\begin{equation}\label{eq. todd class of E_S}
Td(E_\cS)=\prod_{V\in\cS\cup\{\bC^n\}}Td(\bP(\bC^\cS _V)^{\cG^\cS
_V}).
\end{equation}
More precisely, $Td(E_\cS)$ is the product of the pullbacks of
$Td(\bP(\bC^\cS _V)^{\cG^\cS _V})$ under the projections
associated to the decomposition in Proposition \ref{prop.
decomposition of canonical resolution}.
\end{lem}

For every $V\in\cG-\{0\}$ define a formal power series $F_V\in\bZ
[[c_V]]_{V\in\cG}$ by
$$
F_V:=(1-\mathop{\sum_{ W\varsubsetneq
V}}_{W\in\cG}c_W)^{-(n-\delta(V))}(1+c_V)(1-\mathop{\sum_{
W\subset V }}_{W\in\cG} c_W)^{n-\delta (V)}.$$  Also, set
$F_0=(1-c_0)^n$.

\begin{prop}\label{prop. chern class canonical resolution} With the
notation as above, the total Chern class $c(\wti{Y})$ is the image
in $H^*(\wti{Y},\bZ)$ of $\prod_{V\in\cG}F_V$ under the map
(\ref{eq. isom cohomology}).
\end{prop}
\begin{proof} For $V\in \cG-\{0\}$ with $\delta(V)=i+1$, let
$N_{V,i}$ be the normal bundle of $C_{V,i}$ in $Y_i$. Let
$$L_{V,i}=\cO_{Y_i} (E_{0,i}-\mathop{\sum _{0\ne W\varsubsetneq
V}}_{W\in\cG} E_{W,i})).$$ By \cite{DP}-5.1 (the statement in {\it loc. cit.}
needs to be adjusted for the projective case as in Remark \ref{rem. isom}),
$$N_{V,i}\cong {{L_{V,i}}^{\oplus n-1-i}}_{| C_{V,i}}.$$
We want to apply Proposition \ref{prop. chern classes blow up}.
One of the quantities we need is
\begin{align*}
\left [\sum_{0\le k\le n-1-i}\rho_i ^* c_k({L_{V,i}}^{\oplus (
n-1-i)} )\right ]^{-1} &=\rho_i^*c(L_{V,i})^{-(n-1-i)} \\
&=( 1+ [E_{0,i+1}]- \mathop{\sum _{0\ne W\varsubsetneq
V}}_{W\in\cG} [E_{W,i+1}])^{-(n-1-i)}.
\end{align*}
Also, we have
\begin{align*}
&\sum_{0\le j\le n-1-i} (1-[E_{V,i+1}])^j \rho_i ^*
c_{n-1-i-j}(L_{V,i}^{\oplus\ n-1-i})=&\\
&=\sum_{0\le j\le n-1-i} (1-[E_{V,i+1}])^j {n-1-i \choose n-1-i-j}
\rho_i^* c_1(L_{V,i})^{n-1-i-j} \\
&= (1-[E_{V,i+1}] + [E_{0,i+1}] - \mathop{\sum _{0\ne
W\varsubsetneq V}}_{W\in\cG} [E_{W,i+1}])^{n-1-i}
\end{align*}
By Proposition \ref{prop. chern classes blow up},
\begin{align*}
c(Y_{i+1}) = \rho _i ^* c(Y_i)\mathop{\prod _{V\in\cG}}_{\delta
(V)=i+1}\left \{ ( 1+ [E_{0,i+1}]- \mathop{\sum _{0\ne
W\varsubsetneq V}}_{W\in\cG}
[E_{W,i+1}])^{-(n-1-i)} (1+[E_{V,i+1}]) \cdot \right. \\
\left. (1-[E_{V,i+1}] + [E_{0,i+1}] - \mathop{\sum _{0\ne
W\varsubsetneq V}}_{W\in\cG} [E_{W,i+1}])^{n-1-i} \right \}
\end{align*}
Since $\wti{Y}=Y_{n-2}$, the Proposition follows from the last
formula.
\end{proof}

Let $Q(x)=x/(1-\exp (-x))$. For every $V\in\cG-\{0\}$ define a
formal power series $G_V^\cG\in\bQ [[c_V]]_{V\in\cG}$ by
$$G_V^\cG:=Q(-\mathop{\sum_{ W\varsubsetneq
V}}_{W\in\cG}c_W)^{-r(V)}Q(c_V)Q(-\mathop{\sum_{ W\subset
V}}_{W\in\cG} c_W)^{r(V)}.$$ Also, set
$G_0^\cG=Q(-c_0)^n=Q(-c_0)^{r(0)}$. Recall from introduction that
the codimension function $r$ depends only $\cG$, a fact which is
suppressed from the notation. Since the Todd class, as the total
Chern class, is multiplicative on exact sequences of vector
bundles, by Proposition \ref{prop. chern class canonical
resolution} we have:

\begin{cor}\label{cor. todd class canonical resolutions} With the
notation as in Proposition \ref{prop. chern class canonical
resolution}, the Todd class $Td(\wti{Y})$ is the image in
$H^*(\wti{Y},\bQ)$ of $\prod_{V\in\cG}G_V^\cG$ under the map
induced by (\ref{eq. isom cohomology}) after $\otimes_\bZ\bQ$.
\end{cor}

Replacing, in Corollary \ref{cor. todd class canonical
resolutions}, $\wti{Y}=\bP (\bC^n)^\cG$ with $\bP
(\bC_V^\cS)^{\cG^\cS_V}$, we obtain:

\begin{cor}\label{cor. todd class special can res} With the
notation as in Proposition \ref{prop. decomposition of canonical
resolution} and Corollary \ref{cor. todd class canonical
resolutions}, let $\cS\subset\cG-\{0\}$ be a nested subset and let
$V\in\cS\cup\{\bC^n\}$. The Todd class $Td(\bP
(\bC_V^\cS)^{\cG^\cS_V})$ is the image in $H^*(\bP
(\bC_V^\cS)^{\cG^\cS_V},\bQ)$ of
$$\prod_{W'\in\;\cG^\cS
_V}G_{W'}^{\cG_V^\cS}\ \ \ \in\bQ[[c_{W''}]]_{W''\in \cG_V^\cS}$$
under the map $c_{W''}\mapsto [E_{W''}]$ ($W''\ne 0$) and
$c_0\mapsto -[E_0]$.
\end{cor}

Next lemma puts together some computations from \cite{DP}-4.3,
\cite{MS}-Propositions 2.8 and 2.9:

\begin{lem}\label{lem. pullback of divisors} With the notation as in Proposition
\ref{prop. decomposition of canonical resolution}, let
$\emptyset\ne\cS\subset\cG-\{0\}$ be a nested subset.  For
$V\in\cS\cup \{\bC^n\}$, let $p_V$ be the projection of $E_\cS$
onto the factor $\bP(\bC^\cS_V)^{\cG^\cS_V}$ associated to the
decomposition in Proposition \ref{prop. decomposition of canonical
resolution}. Let $W'\in\cG^\cS_V$ with corresponding divisor
$E_{W'}'$ in $\bP (\bC_V^\cS)^{\cG^\cS_V}$.

(a) If $W'\ne 0$, then $p_V^* E_{W'}' \sim E_W\,_{|E_\cS}$, where
$W$ is the unique element of $\cG$ nested between $V$ and $V_\cS$
whose image in $\bC^\cS_V=V/V_\cS$ is $W'$.

(b) If $W'=0$, then
$$p_V^* E_0' \sim \left (E_0 -  \mathop{\sum_{0\ne W\subset V_\cS, W\in\cG}}_{\{W\}\cup S\text{ nested}} E_W\right )_{|E_\cS}.$$
\end{lem}
\begin{proof} For $\cS$ having only one element, this is
\cite{MS}- Proposition 2.8. For the rest, one iterates as in
\cite{MS}- Proposition 2.9 or, equivalently, as in the last
paragraph of the proof of the theorem of \cite{DP}-4.3.
\end{proof}

\begin{prop}\label{prop. todd class of E_S} With the notation as
in Proposition \ref{prop. decomposition of canonical resolution}
and Definition \ref{def. poly P}, $Td(E_\cS)$ is the image of the
formal power series
$$T^\cS:=\prod_{V\in\cS\cup\{\bC^n\}}\ \ \mathop{\prod_{V_\cS\subset
W\varsubsetneq V}}_{W\in\cG} P_W^{\cS , V}\ \ \ \ \ \in
\bQ[[c_W]]_{W\in\cG}
$$
in $H^*(E_\cS,\bZ)$ under the map $l_\cS: 1\mapsto
[\wti{Y}]_{|E_\cS}, c_{W}\mapsto [E_{W}]_{|\, E_\cS}$ ($W\ne 0$),
and $c_0\mapsto -[E_0]_{|\,E_\cS}$.
\end{prop}
\begin{proof} For $V\in\cS\cup \{\bC^n\}$, let $p_V$ be the projection of $E_\cS$ onto the
factor $\bP(\bC^\cS_V)^{\cG^\cS_V}$ associated to the
decomposition in Proposition \ref{prop. decomposition of canonical
resolution}. Define a map of $\bQ$-algebras
\begin{align*}
p_V^* : \bQ [[c_{W'}]]_{W'\in\cG_V^\cS} \longrightarrow
\bQ[[c_W]]_{W\in\cG},
\end{align*}
as follows. For $W'\ne 0$, let $c_{W'} \mapsto c_W$, where
$W=\pi^{-1}(W')$ and $\pi: V \twoheadrightarrow \bC_V^\cS$. For
$W'=0$, let
$$c_0 \mapsto \mathop{\sum_{ W\subset V_\cS, W\in\cG}}_{
\{W\}\cup S \text{ nested} } c_W.$$ By Lemma \ref{lem. pullback of
divisors}, we have a commutative diagram of $\bQ$-algebras

$$\xymatrix{
\bQ [[c_{W'}]]_{W'\in\cG_V^\cS} \ar[r]^{p_V^*} \ar[d] & \bQ[[c_W]]_{W\in\cG} \ar[d]^{l_\cS} \\
H^*(\bP(\bC^\cS _V)^{\cG_V ^\cS},\bQ) \ar[r]^{\ p_V^*}&
H^*(E_\cS,\bQ).}
$$

For $W'\in \cG_V^\cS$, denote by $\bar{W'}$ the subspace
$\pi^{-1}(W')$ of $V$, where $\pi:V \twoheadrightarrow \bC^\cS_V$.
Then $p_V^*G_{W'}^{\cG^\cS_V}=P_{\bar{W'}}^{\cS,V}$. Then the
Proposition follows from Lemma \ref{lem. todd class of E_S} and
Corollary \ref{cor. todd class special can res}.
\end{proof}

\medskip
\noindent {\bf Proof of Lemma \ref{lem Computation of chi E_S}.}
Let $\cE$ be the invertible sheaf $\cO_{E_{\cS}}
(\sum_{V\in\cG}a_V'(c)E_V )$. By definition, for $i\ge 0$,
$ch_i(\cE)=(1/i!)(\sum_{V\in\cG} a_V'(c)[E_V]_{|\,E_\cS} )^i$.
Theorem \ref{thm HRR} allows us to write
\begin{align*}
 \chi(E_\cS,\cE) = \sum_{i+j=n-1-|\cS|}\frac{1}{i!}\left (\sum_{V\in\cG} a_V'(c)[E_V]_{|\,E_\cS}
 \right )^i\cdot Td_j(E_\cS).
\end{align*}
The Lemma follows from the map (\ref{eq. isom cohomology}) and
Proposition \ref{prop. todd class of E_S}. Remark that the map
$l_\cS$ from Proposition \ref{prop. todd class of E_S}, factors on
homogenous polynomials of degree $n-1-|\cS|$ via: multiplication
of the map (\ref{eq. isom cohomology}) with $\prod_{V\in\cS}c_V$.
\begin{flushright}
$\Box$
\end{flushright}

\medskip
\noindent {\bf Proof of Theorem \ref{thm. inner jumping numbers}.}
By \cite{Bu}- Proposition 2.7 (ii),
$$n_{c,x}(D)=\chi(\wti{Y}, \cO_{E^{\cS_{c,x}}}(K_{\wti{Y}/Y}-\rndown{(c-\epsilon)\rho^*
D})),$$ for $0<\epsilon \ll 1$, where $\cS_{c,x}$ is empty unless
$cd\in\bZ$ and there exists a divisor on $\wti{Y}$ mapping onto
$\{x\}$. In the later case, $\cS_{c,x}=\{V_x\}$. Thus
$$n_{c,x}(D)=\chi(\cO_{E^{\cS_{c,x}}}(\sum_{V\in\cG-\{0\}}a_V(c)E_V),$$
and the Theorem follows by replacing $a_0$ with $0$ in the proof
of Lemma \ref{lem Computation of chi E_S}.
\begin{flushright}
$\Box$
\end{flushright}

\end{section}

\begin{section}{Examples}

The following  examples illustrate how Theorems \ref{thm. jumping
numbers} and \ref{thm. inner jumping numbers} work.

\medskip
\noindent (a) Let $D$ be the union of three distinct lines
passing through one point in $\bP^2$. Let $\cA=\{V_1,V_2,V_3 \}$,
$V_i\subset\bC^3$ mutually distinct subspaces of dimension 2, with
$V_1\cap V_2\cap V_3=L$ where $\delta (L)=1$. Then
$D=\bP(V_1)+\bP(V_2)+\bP(V_3)$ as a divisor in $\bP(\bC^3)=\bP^2$.

Take $\cG=\{0,L,V_1,V_2,V_3\}$. By (\ref{eq, type 2}),
$c_0+c_L+c_{V_i}$ ($i=1,2,3$) belongs to the ideal $I$. We can
eliminate thus the variables $c_{V_i}$ ($i=1,2,3$) and have
$$\bQ[c_V]_{V\in\cG}/I\cong
\bQ[c_0,c_L]/(c_0^3,c_0c_L,(c_0+c_L)^2),$$ and, by (\ref{eq. isom
cohomology}), this is isomorphic with the cohomology ring of
$\wti{Y}$, the blow up of $\bP^2$ at $\bP(L)$.

The only $c\in (0,1)$ for which $\cS_c\ne \emptyset$ are $c=1/3,
2/3$. For both cases, $\cS_{c}=\{L\}$; call this set $\cS$. We
have
$$T^{\cS}=P_0^{\cS,L}P_L^{\cS,\bC^3}\prod_{1\le i\le
3}P_{V_i}^{\cS,\bC^3}.$$ From the fact that $Q(x)=1+\frac{1}{2}x
+\text{(degree }\ge 2\text{ terms)}$, we get
$$T^\cS=1+(-\frac{3}{2}c_0-c_L)+ \text{(degree }\ge 2\text{ terms)}.$$
It follows by Theorem \ref{thm. jumping numbers} that $c=1/3$ is a
jumping number for $D$ if and only if $-\frac{5}{2}c_0c_L$ does
not lie in the ideal $I\subset\bQ[c_V]_{V\in\cG}$. Also, $c=2/3$
is a jumping number if and only if $-\frac{5}{2}c_0c_L+c_L^2$ does
not lie in $I$. Therefore $c=\frac{2}{3}$ is the only jumping
number of $D$ in the interval $(0,1)$.
By Theorem \ref{thm. inner jumping numbers}, the inner jumping multiplicity at $x=\bP(L)$ of $c=2/3$ is given by writing $-\frac{3}{2}c_0c_L-c_L^2\in\bQ[c_V]_{V\in\cG}/I$ in terms of $c_0^2$. Thus $n_{x,\frac{2}{3}}(D)=1$. Also, $n_{x,1}(D)$ is given by $-\frac{3}{2}c_0c_L-2c_L^2$, hence $n_{x,1}(D)=2$. This gives the initial part of the spectrum of $D$ at $x$, and in fact (in this case by symmetry) the whole spectrum is $t^{2/3}+2t+t^{4/3}$.

\medskip
\noindent (b) Consider the central hyperplane arrangements in
$\bC^3$ given by
$$
(x^2-y^2)(x+z)(x+2z),$$
$$(x^2-y^2)(x^2-z^2).
$$
They are combinatorially equivalent. To apply the algorithm of
this article, we consider the completion of these arrangements to
$\bP^3$. Here $\cA=\{A_i\subset\bC^4\ |\  i=1,\ldots , 4\}$, and
$\cG=L(\cA)-\{\bC^4\}$ is given by
$$\{0,C,B_1,\ldots ,B_6,A_1,\ldots, A_4\},$$
where $C, B_j, A_i$ have codimension $3, 2,$ resp. $1$, $C$ is
included in all $B_j$, and $B_j\subset A_i$ if $(i,j)$ lies in
$$
M:=\{
(1,1),(1,2),(1,3),(2,2),(2,5),(2,6),(3,1),(3,4),(3,6),(4,3),(4,4),(4,5)
\}.
$$
The ideal $I$ is generated by $c_{A_i}+\sum_{(i,j)\in
M}c_{B_j}+c_C+c_0$, $c_0c_C$, $c_{B_j}c_{B_{j'}}$ with $j\ne j'$,
$c_{B_j}(c_0+c_C)$, and $c_{B_j}^2+c_0^2+c_C^2$. The only nonempty
$\cS_c$ for $c\in (0,1)$ are $\cS_{1/4}=\{C\}$,
$\cS_{2/4}=\{C,B_1,\ldots,B_6\}$, $\cS_{3/4}=\{C\}$. Then, modulo
$I$, we have
$$
T^{\{C\}}=-\frac{2}{3}c_0^3+c_C^2+\frac{11}{6}c_0^2+\frac{1}{4}(c_{B_1}+\ldots
+c_{B_6})c_0-\frac{1}{2}(c_{B_1}+\ldots
+c_{B_6})-\frac{3}{2}c_C-2c_0+1,
$$
$$
T^{\{C,B_j\}}=-\frac{5}{8}c_0^3+\frac{11}{2}c_C^2+\frac{1}{2}c_{B_j}c_0+\frac{7}{4}c_0^2-c_{B_j}-
\frac{3}{2}c_C-2c_0+1,
$$
$$
T^{\{B_j\}}=-c_0^3+\frac{1}{4}c_C^2+c_{B_j}c_0+\frac{7}{4}c_0^2-c_{B_j}-c_C-2c_0+1.
$$
One computes using Theorem \ref{thm. jumping numbers} that $1/4$
and $2/4$ are not jumping numbers, but $3/4$ is the only jumping
number in $(0,1)$. Using Theorem \ref{thm. inner jumping numbers},
one computes that the inner jumping multiplicities of $1/4$ and
$2/4$ are $0$, whereas the inner jumping multiplicities of $3/4$
and $1$ are $1$, and resp. $3$. By \cite{Bu}, these are the same
as the spectrum multiplicities. We used Macaulay 2 for some of the
computations above.

The jumping numbers in this case can be computed directly from
\cite{Te} - Lemma 2.1 (see also Lemma \ref{lemma reduction to
euler char} here) and the result agrees with ours. The spectrum in
this case can be computed by \cite{St} -Theorem 6.1 which treats
the case of homogeneous polynomials with 1-dimensional critical
locus, and the beginning part agrees with what we have found.
Remark that there is a shift by multiplication by $t$ between the
definition of spectrum of \cite{St} and that of \cite{Bu}.

\end{section}

\end{document}